\documentclass[11pt,reqno]{amsart}

\usepackage{amsmath,amsthm,amssymb,comment,fullpage}
\usepackage{braket}
\usepackage{mathtools}
\usepackage{multirow}
\usepackage{trfsigns}

\usepackage{caption}
\usepackage{times}
\usepackage[T1]{fontenc}
\usepackage{mathrsfs}
\usepackage{latexsym}
\usepackage[dvips]{graphics}
\usepackage{epsfig}
\usepackage{amsmath,amsfonts,amsthm,amssymb,amscd}
\input amssym.def
\input amssym.tex
\usepackage{color}
\usepackage{hyperref}
\usepackage{setspace}
\usepackage{url}
\newcommand{\bburl}[1]{\textcolor{blue}{\url{#1}}}

\usepackage{tikz}
\usepackage{tkz-tab}
\usepackage{tkz-graph}
\usetikzlibrary{shapes.geometric,positioning}

\newcommand{\burl}[1]{\textcolor{blue}{\url{#1}}}

\newcommand{\mattwo}[4]
{\left(\begin{array}{cc}
                        #1  & #2   \\
                        #3 &  #4
                          \end{array}\right) }

\numberwithin{equation}{section}

\newtheorem{thm}{Theorem}[section]

\newtheorem{cor}[thm]{Corollary}
\newtheorem{lem}[thm]{Lemma}

\newtheorem{defi}[thm]{Definition}

\theoremstyle{plain}

\newtheorem{rem}[thm]{Remark}

\newcommand\be{\begin{equation}}
\newcommand\ee{\end{equation}}
\newcommand\bee{\begin{equation*}}
\newcommand\eee{\end{equation*}}
\newcommand\bea{\begin{eqnarray}}
\newcommand\eea{\end{eqnarray}}
\newcommand\beae{\begin{eqnarray*}}
\newcommand\eeae{\end{eqnarray*}}
\newcommand\bi{\begin{itemize}}
\newcommand\ei{\end{itemize}}
\newcommand\ben{\begin{enumerate}}
\newcommand\een{\end{enumerate}}
\newcommand\bc{\begin{center}}
\newcommand\ec{\end{center}}
\newcommand\ba{\begin{array}}
\newcommand\ea{\end{array}}

\newcommand{\C}{\ensuremath{\mathbb{C}}}
\newcommand{\Z}{\ensuremath{\mathbb{Z}}}

\newcommand\frakfamily{\usefont{U}{yfrak}{m}{n}}
\DeclareTextFontCommand{\textfrak}{\frakfamily}

\newcommand{\supp}{\operatorname{supp}}
\newcommand{\sinc}{\operatorname{sinc}}
\newcommand{\rect}{\operatorname{rect}}

\begin{document}
\title{Upper Bounds for the Lowest First Zero in Families of Cuspidal Newforms}

\author[Tang]{Xueyiming Tang}
\email{\textcolor{blue}{\href{mailto:emily_tang@loomis.org}{emily$\_$tang@loomis.org}}},
\address{The Loomis Chaffee School, Windsor, CT 06095, USA}


\author[Miller]{Steven J. Miller}
\email{\textcolor{blue}{\href{mailto:sjm1@williams.edu}{sjm1@williams.edu}},
\textcolor{blue}{\href{Steven.Miller.MC.96@aya.yale.edu}{Steven.Miller.MC.96@aya.yale.edu}}}
\address{Department of Mathematics and Statistics, Williams College,
Williamstown, MA 01267, USA}



\date{}

\begin{abstract}
\onehalfspacing
Assuming the Generalized Riemann Hypothesis, the non-trivial zeros of $L$-functions lie on the critical line with the real part $1/2$. We find an upper bound of the lowest first zero in families of even cuspidal newforms of prime level tending to infinity. We obtain explicit bounds using the $n$-level densities and results towards the Katz-Sarnak density conjecture. We prove that as the level tends to infinity, there is at least one form with a normalized zero within $1/4$ of the average spacing. We also obtain the first-ever bounds on the percentage of forms in these families with a fixed number of zeros within a small distance near the central point.
\end{abstract}

\maketitle

\tableofcontents


\section{Introduction}

\subsection{Background}

The zeros and poles of $L$-functions carry substantial information about various quantities in number theory, from the distribution of primes to class numbers to the Mordell-Weil group (at least conjecturally \cite{BSD1, BSD2}). The starting point of many of these investigations is the Generalized Riemann Hypothesis (GRH): all non-trivial zeros in the critical strip have real part 1/2. As the distribution of these zeros is of great importance, a natural question is the following: \emph{assuming GRH, how far up must we go on the critical line before we are assured that we will see the first zero?}

Several studies have been done on the first zero of $L$-functions. The first critical zero of the Riemann zeta-function is surprisingly large compared to other $L$-functions, around $\frac{1}{2} + 14.135i$. Stephen D. Miller \cite{M} proved an upper bound on the first zero of the form $\frac{1}{2} + i\gamma$ for all $L$-functions of real archimedian type of $-\gamma_0 < \gamma < \gamma_0$, where $\gamma_0 \approx 14.13$. This is generalized in \cite{B--} to hold for general $L$-functions, for $-\gamma_1 < \gamma < \gamma_1$ with $\gamma_1 \approx 22.661$; see also \cite{B, HR}. For elliptic curves, Mestre \cite{Mes} proved the first zero occurs by $O(1 / \log \log N_E)$, where $N_E$ is the conductor of $E$, though we expect the correct scale to study the zeros near the central point is the significantly smaller $1 / \log N_E$.

Complementing the previous studies which gave bounds for the lowest first zero for individual $L$-functions, we explore what improvements arise when we examine all forms in a family and investigate how high up we must go to ensure we see the lowest of the lowest zeros. We prove upper bounds on the lowest first zero above the central point among all the forms for the family of even cuspidal newforms of prime level tending to infinity; similar calculations work for the family of odd cuspidal newforms with a trivial modification (since odd forms have a zero at the central point by symmetry, one simply removes that known contribution from the computations). Moreover, we obtain results on the number of zeros near the central point (low-lying zeros) by obtaining an upper bound on the percentage of forms in the family that have a certain number of zero in a small interval near the central point. We bound the percentages of forms in these families that can have ``many'' zeros ``close'' to the central point (the smaller the window and the greater the number of zeros the less likely this is to happen; we quantify how unlikely).

Below we describe the families of $L$-functions we study, and then state our results. We assume the reader is familiar with the literature; see \cite{IK} for a review of needed properties of $L$-functions, and \cite{BFMT-B, KS2} for a history and results on connections between the behavior of zeros near the central point and eigenvalues of random matrix ensembles. The methods of this paper are an extension of the arguments in \cite{GM,Mes}.

\subsection{Modular Forms}

We quickly review some background material and notation. Given a matrix $\gamma = \mattwo{a}{b}{c}{d}$ and a complex number $z$, we set $\gamma z$ to be $\frac{az + b}{cz + d}$. The matrices $T$ and $S$, with respective actions \begin{equation}
 \mattwo{1}{1}{0}{1} z \ = \ z + 1, \ \ \ \mattwo{0}{-1}{1}{0} z \ = \ -1/z,
\end{equation} generate the group ${\rm SL}_2(\Z)$, the set of $2 \times 2$ integer matrices with determinant 1; below we look at functions that transform nicely under this group, as well as subgroups of this group.

\begin{defi}
Let $f$ be a function on the upper-half of the complex plane $\mathcal{H} = {z \in \C : \Im(z) > 0}$. A modular form of weight $k$ is an meromorphic function $f: \mathcal{H} \to \C$ such that \begin{equation}
    f(\gamma z) \ = \ (cz+d)^k f(z)
\end{equation}
for all $\gamma\ =\ \left(\begin{smallmatrix}
    a & b\\
    c & d
\end{smallmatrix}\right) \in SL_2(\Z)$ (the set of two by two integer matrices with  determinant 1). Additionally, $f$ is bounded as $\Im(z)$ goes to infinity.
\end{defi}

Notice that if we take $\gamma = \left(\begin{smallmatrix}
1 & 1\\
0 & 1
\end{smallmatrix}\right) \in SL_2(\Z)$, we have $f(z + 1) = f(z)$, which means that $f$ is periodic with period 1. The cusps forms are modular forms where $f(z)$ approaches $0$ when $\Im(z) \to \infty$.

We can generalize the above definition through additional restrictions on matrices under which our functions should transform nicely. A modular form of level $N$ has the additional property that we only care about the transformation rule for $f(\gamma z)$ when the lower left entry of $\gamma$ (what we called $c$ above) equals zero modulo $N$.

\begin{defi}[Cuspidal Newforms]
Let $H^\star_k(N)$ be the set of holomorphic cusp forms of weight $k$ that are newforms\footnote{This means they are not lifted trivially from a modular form of level $M|N$.} of level $N$. For every $f\in H^\star_k(N)$, we have a Fourier expansion
    \begin{equation}
        f(z)\ = \ \sum_{n=1}^\infty a_f(n) e(nz)
    \end{equation} with $e(z) = e^{2\pi i z}$. We set $\lambda_f(n) =  a_f(n) n^{-(k-1)/2}$. The $L$-function associated to $f$ is
    \begin{equation}\label{eq: L-function}
        L(s,f)\ =\ \sum_{n=1}^\infty \lambda_f(n) n^{-s}.
    \end{equation}
The completed $L$-function is
    \begin{equation}\label{eq:completed_L_func}
        \Lambda(s,f) \ :=\ \left(\frac{\sqrt{N}}{2\pi}\right)^s \Gamma\left(s+\frac{k-1}{2}\right) L(s,f).
    \end{equation}
\end{defi}

Each $\Lambda(s,f)$ satisfies a functional equation $\Lambda(s,f) = \epsilon_f \Lambda(1-s,f)$, where $\epsilon_f$ is either +1 or -1. We can thus separate $H^\star_k(N)$ into two subsets: $$H^+_k(N)\ =\ { f\in H^\star_k(N): \epsilon_f = +1} \ \ \  {\rm and} \ \ \  H^-_k(N) \ =\ { f\in H^\star_k(N): \epsilon_f = -1}.$$ The first set are the even forms (which we study below), while the second are the odd ones. We adjusted the coefficients of our forms by removing the factor $n^{-(k-1)/2}$ so that the functional equation relates $s$ to $1-s$.

Since the work of Montgomery and Dyson \cite{Mon}, there is now an extensive literature that many properties of $L$-functions, from their values to the distribution of spacings between their zeros, is well modeled by ensembles of matrices (for details see \S\ref{subsec:rmt} as well as the references mentioned earlier). The symmetry group corresponding to $H^\star_k(N)$ is the Orthogonal group, denoted as ${\rm O}$. For the subset $H^+_k(N)$, its associated symmetry group is scaling limit of the Special Orthogonal group of even sized matrices, represented as ${\rm SO(even)}$, and for $H^-_k(N)$, its associated symmetry group is the scaling limit of the Special Orthogonal group of odd sized matrices, represented as ${\rm SO(odd)}$.

\subsection{Random Matrix Theory}\label{subsec:rmt}

Montgomery's work in 1972 unveiled a fascinating link between the distribution of zeros of the Riemann zeta-function and the distribution of eigenvalues of random Hermitian matrices.
He conjectured that the $n$-level correlation function of the zeta function's zeros and that of the Gaussian Unitary Ensemble (GUE) coincide; for suitably restricted test functions this was later proven by Hejhal \cite{Hej} for the 3-level of the Riemann zeta function and Rudnick and Sarnak \cite{RS} for all $n$ for all automorphic $L$-functions). As our goal is to study other statistics than the $n$-level correlations, we just remark that this was the first statistic studied, and knowing these correlations for all $n$ yields the spacing distribution between adjacent zeros. Odlyzko calculated the zeros of the zeta-function at the height of the $10^{20{\rm th}}$ zero and observed that main term of the spacing of nearby zeros agrees with that of the GUE (see \cite{H, BFMT-B, Con} and their references for more details on these connections).

The reason we move to other statistics is that although the $n$-level correlations provide amazing results of the distribution of zeros of one form, these are defined through limits of sums over the zeros on the critical line and are thus insensitive to finitely many zeros. In particular, they cannot say anything about the behavior near the central point, which is often of great arithmetic interest. Katz and Sarnak \cite{KS1, KS2} then observed that the statistics of zeros near the central point of many $L$-functions also align with the eigenvalues near 1 of classical compact groups, and thus the GUE is not the full story. Yhey introduced a statistic, the $n$-level densities, where unlike the correlations most of the contribution is from the zeros very close to the central point and it is thus very sensitive to finitely many zeros.

\begin{defi}[$n$-Level Density]\label{nleveldensity}
The $n$-level density of an $L$-function $L(s,f)$ for a \emph{test function} $\Phi: \mathbb{R}^n \to \mathbb{R}$ is defined as
\begin{equation}\label{def:density}
    D_n (f; \Phi) \ := \  \sum_{\substack{j_1, \cdots, j_n \\ j_i \neq \pm j_k}} \Phi \left( \frac{\log c_f}{2\pi} \gamma_{f, j_1}, \dots, \frac{\log c_f}{2\pi} \gamma_{f, j_n} \right),
\end{equation}
where $c_f$ is the analytic conductor of $f$ and the non-trivial zeros\footnote{We are not assuming} of $L(s, f)$ are denoted by $\frac{1}{2} + \gamma_f^{(j)}$. In many applications we assume $\Phi$ factors as a product of $n$ copies of a function $\phi$: $\Phi(x_1, \dots, x_n) = \prod_{j=1}^n \phi(x_j)$. We denote a zero scaled by the logarithm of the analytic conductor by a tilde; thus $\widetilde{\gamma}_f^{(j)} = \gamma_f^{(j)} \frac{\log c_f}{2\pi}$.
\end{defi}

The analytic conductor arises from the functional equation of the $L$-function, and near the central point the spacing between adjacent zeros is on the order of the reciprocal of its logarithm. This thus provides the natural scale to study the distance of zeros from the central point; on average it is around $1/\log c_f$ from one such zero to another; we are interested in how often the first or first few zeros are less than a fixed percentage of this quantity. In other words, how often is the first zero less than half the average spacing? For families of even $N$ cuspidal newforms, the analytic conductor doesn't depend on the form.

\ \\
\textbf{\emph{NOTE: Whenever we talk about bounds on the location of zeros near the central point, we always mean relative to the average spacing; thus a bound of .25 means .25 times the reciprocal of the logarithm of the analytic conductor.}}\\ \

We must balance between reaching the optimal results and having calculations that can be done in closed form; to deal with these tradeoffs we choose a test function $\Phi$ with certain characteristics that yield very good bounds while still having feasible calculations. As the zeros are symmetric about the real axis, we need only study $\Phi$ even.\footnote{Any test function is the sum of an even and an odd function, the sum of the odd part over the zeros vanish, so there is no need to include it.} We always choose $\Phi$ to be a Schwartz function, so it and all its derivatives decay faster than $|x|^{-A}$ for any $A>0$.

In order to execute the number theory sums that arise, we need the test function to have compact support for its Fourier transform, defined in \eqref{fouriertransform}. If we could have the test function concentrated near the central point and rapidly decaying, we would glean a significant amount of information on what is happening there; for example, if we could take a delta spike then the only contribution would be from zeros at the central point. Sadly, just like the Heisenberg Uncertainty Principle that illustrates the trade-off between knowing a particle's position and momentum, we cannot have both of the above conditions fulfilled simultaneously. A more concentrated test function means a larger support of its Fourier transform, and a more compact support of the Fourier transform results in less decay for the test function and thus potentially larger contributions from zeros away from the central point. To investigate the behavior of the first (or first few) zeros above the central point, we impose additional requirements for the test function, with the conditions varying depending on what we are trying to prove.

For Theorem \ref{evenmainresult}, we use the naive test function, defined as follows. In all applications it will be clear what $\sigma_n$ is: it is the largest support for which we know the $n$-level density or $n$\textsuperscript{th} centered moments. To indicate this dependence on $\sigma_n$ we include a subscript $n$ below.

\begin{defi}[Naive Test Function]\label{defi:naive} The naive test function for the $n$-level is defined by
\begin{equation}\label{naive}
    \phi_{{\rm naive}; n}(x) \ := \ \left(\frac{\sin(\pi \sigma_nx)}{\pi \sigma_n x}\right)^2,
\end{equation} and its Fourier transform is \begin{equation}
    \widehat{\phi}_{{\rm naive}; n}(y)\ =\ \frac{1}{\sigma_n}\left(y-\frac{|y|}{\sigma_n}\right)
\end{equation} for $|y| < \sigma_n$ and zero otherwise.
\end{defi}

Since the naive test function is non-negative, it is useful in determining upper bounds as we can drop non-negative contributions of the density. Some previous studies such as \cite{BCDMZ, FM, ILS} attempt to find the optimal non-negative test functions for certain level densities; however, since the improvement in results is quite small, we are satisfied with just using the naive test function as it greatly simplifies the exposition, though with additional work slight improvements are possible.

For Theorem \ref{oddmainresult}, we have a different set of characteristics for the test function $\phi$. Instead of being completely non-negative, we now only want it to be non-negative within a certain distance from the central point, and non-positive beyond (and $\phi(0) \neq 0$). Such a setup allows us to drop non-positive contributions to the density, thus getting a lower bound. To find such $\phi$, it is easier to first construct its finitely supported Fourier transform; we provide details in \S\ref{construction}.

\subsection{Centered Moments}
The $n$\textsuperscript{th} centered moments are an alternative, but essentially equivalent after some combinatorics, statistics to the $n$-level densities. Hughes-Miller \cite{HM} show their expansions are better suited to obtain bounds than the determinant expansions of Katz-Sarnak. For the following theorem from \cite{C--}, we are only able to take the $\sigma$ up to $2$, though the Random Matrix Theory correspondence suggests that we can take arbitrarily large support. We state the results below for general $\sigma$, so improvements in those calculations immediately transfer to improvements here.

\begin{thm}\label{bound} Assume GRH.
    Let $n \geq 2$ and ${\rm supp}(\phi) \subset (-\frac{\sigma}{n}, \frac{\sigma}{n})$, where by \cite{C--} we may take $\sigma = 2$. Define
    \begin{equation}
      \sigma_{\phi}^2 \ := \ 2\int_{-\infty}^{\infty} |y|\widehat{\phi}(y)^2 dy
    \end{equation}
    and
\begin{multline}
    R(m, i; \phi) \ := \ 2^{m-1}(-1)^{m+1}\sum_{l=0}^{i-1} (-1)^l \binom{m}{l} \\ \left(-\frac{1}{2}\phi^m(0)  +\int_{-\infty}^{\infty} \cdots \int_{-\infty}^{\infty} \widehat{\phi}(x_2) \cdots \widehat{\phi}(x_{l+1}) \right. \\ \left. \int_{-\infty}^{\infty} \phi^{m-l}(x_1)\frac{\sin(2\pi x_1(1+|x_2|+\cdots+|x_{l+1}|))}{2\pi x_1}dx_1 \cdots dx_{l+1} \right)
\end{multline}
and
\begin{equation}
    S(n, a, \phi) \ := \ \sum_{l=0}^{\lfloor\frac{a-1}{2}\rfloor} \frac{n!}{(n-2l)!l!}R(n-2l,a-2l,\phi)\left(\frac{\sigma_{\phi}^2}{2}\right)^l.
\end{equation} By $\langle Q(f) \rangle_{N;\pm}$ we mean the average of $Q(f)$ over all $f$ in the family of even (odd) cuspidal newforms of level $N$ for the positive (negative) sign (the number of such forms is proportional to $N$).
Then
\begin{equation}\label{RHSlimit}
    \lim_{\substack{N\to\infty \\ N \text{\rm prime}}} \langle \left(D(f;\phi) - \langle D(f;\phi) \rangle_{N;\pm} \right)^{n}\rangle_{N;\pm}  \ = \ 1_{n \ \rm even}(n-1)!! \sigma_{\phi}^n \pm S(n, a; \phi),
\end{equation}
for
\[
    1_{n \ \rm even}(n) \ = \
    \begin{cases}
      1 & \emph{if } n = 2m \emph{ is even} \\
      0 & \emph{if } n = 2m + 1 \emph{ is odd.}
   \end{cases}
\]
\end{thm}

\subsection{Main Results}
We obtain results on the number and location on low-lying zeros of the family of cuspidal newforms of level $N$ through $n$\textsuperscript{th} centered moment. In Theorem \ref{oddmainresult} we give an upper       bound $(-\omega, \omega)$ on the lowest first zero of the family through odd-centered moments. We prove additional results in Theorem \ref{evenmainresult} for zeros near the central point and provide an upper bound on the percentage of forms in the family with at least $r$ zeros within a certain interval $(-\rho, \rho)$ near the central point.

Previous work focused on using the 1-level density to obtain such bounds; the $n$-level densities and $n$\textsuperscript{th} centered moments are far more difficult to work with, leading to significantly harder integrals as well as combinatorial challenges in removing the contributions of certain forms and isolating the desired results. However, as the support increases these results are significantly better than those from the 1-level, justifying the efforts; other work has shown the value of this approach with better results for bounding the probability of vanishing to order $r$ or greater at the central point; see \cite{DM, LM}. Instead of trying to bound the lowest zero of one form, our question is different as we are searching for the lowest of the low and thus are able to get significantly better results.

\begin{thm}\label{oddmainresult} Assume GRH.
    Let $h(y)$ be an even function that is at least twice continuously differentiable with $h$ supported in $[-1, 1]$, and monotonically decreasing from $0$ to $1$. We define $f(y) = h(2yn/ \sigma$), $g(y) = (f \ast f)(y)$ (convolution of $f$ with $f$). Let $\widehat\phi_{\omega}(y)$, the Fourier transform of $\phi_{\omega}(x)$, equal $g(y) + (2 \pi \omega)^{-2} g''(y)$. Thus we have $\supp (\widehat{\phi}_{\omega}) \subset (-\sigma/n, \sigma/n)$, and $\phi_{\omega}(x)$ non-negative when $|x|\leq \omega$ and non-positive when $|x|>\omega$. Then for an odd $n = 2m + 1$, whenever $\omega$ satisfies the inequality
    \begin{equation}
        -\left( \widehat\phi_{\omega}(0) + \frac{1}{2}\int_{-\sigma/n}^{\sigma/n}\widehat\phi_{\omega}(y)dy\right)^{n} < \ 1_{n \ \rm even}(n-1)!! \sigma_{\phi_{\omega}}^n + S(n, a; \phi_{\omega}),
    \end{equation}
    then there exists at least one form with at least one normalized zero in the interval $(-\omega, \omega)$. In particular, if we denote $(-\omega_{\min},\omega_{\min})$ to be the interval near the central point that contains at least one normalized zero, then the one-level density gives us an explicit expression that there is at least one zero in this interval whenever $\omega_{\min}$ satisfies
    \begin{equation}
        \omega_{\min}(\sigma, h) \ > \ \left( -\frac{\sigma \int_{0}^{1}h(u)^{2}\,du + \frac{\sigma^2}{4}\int_{-1}^{1}\int_{0}^{2/\sigma} h(u) h(v - u) \,dv \,du}{\frac{1}{\sigma}\int_{0}^{1} h(u)h''(u) \,du + \frac{1}{4}\int_{-1}^{1}\int_{0}^{2/\sigma} h(u) h''(v - u) \,dv \,du}\right)^{-\frac{1}{2}} \pi^{-1}.
    \end{equation}
    For the support $\sigma = 2$ (which is the largest known value we can take for this family if we only assume GRH, though it is conjectured to hold for all $\sigma$) and $h = \cos(\pi x/2)$ for $|x| \le 1$ and 0 otherwise (the optimum function to use for one-level density), the smallest first non-trivial normalized zero of the family of even cuspidal newforms has imaginary part $\gamma$ in the interval $[-\gamma_{\min}, \gamma_{\min}]$ with $\gamma_{\min} \approx 0.25$.
\end{thm}

While increasing the level $n$ gives better results, if we can only use $\sigma = 2$ (currently the best known result under the assumption of just GRH) as we compare the bounds arising from using results from levels $1$, $3$ and $5$, we actually see worse results as the level increases. We prove that $\omega$ and $\sigma$ are inversely proportional, which means that if we take a larger $\sigma$, $\phi_\omega$ can be more focused around the central point, better approximating the delta spike, and then the benefit of raising to the power $n$ comes into play and yields smaller intervals. The Random Matrix Theory conjectures imply that we can take $\sigma$ arbitrarily large (though as remarked results are proven in the number theory side, for sums over zeros, under GRH only for $\sigma$ up to $2$, though assuming additional conjectures allows one to increase slightly). Thus, if we use a $\sigma$ greater than $2$, we see that eventually as the support is large enough, the higher levels yield better results.

\begin{rem}
    We examined additional results from level 3 and 5. Unfortunately, higher levels involve difficult integrals to evaluate that do not have closed form evaluations, especially when we are using a complicated test function involving Fourier transforms and convolution. While we cannot use theoretical arguments or software packages such as Mathematica to evaluate in closed form the multi-dimensional integrals on the right-hand side, by using Riemann Sums we can calculate the integrals accurately enough by having the error term sufficiently small, and evaluate the integrals up to a small enough error for our purposes. We found that the result from the 3-level eventually is  better than that from 1-level, for example when $\sigma = 7$, and the 5-level statistics beats both one and three levels when $\sigma = 6$. These results highlight the importance of deriving the theory on bounds arising from higher moments.
\end{rem}

Previous investigations have mostly been concerned with bounds on the first zero above the central point, though Goes-Miller \cite{GM} obtained upper and lower bounds for the number of zeros in one-parameter families of elliptic curves in windows near the central point. We build on their work and bound the percentage of forms with a large number of zeros in a small window about the central point; there is tremendous freedom in assigning values to ``large'' and ``small'' above; in brief, we obtain excellent bounds that the probability is very low of many more zeros in a small window than expected. As the averaging formulas for cuspidal newforms are better than those for elliptic curve $L$-functions, we can do the number theory calculations for larger $\sigma$ and thus obtain better results than in the elliptic curve case. Specifically, we bound the percentage of forms in the family of cuspidal newforms $\mathcal{F}_N$ with $r$ normalized zeros on the interval $(-\rho, \rho)$. If $\rho$ is sufficiently small we expect there to be few such zeros.

\begin{thm}\label{evenmainresult} Assume GRH.
    Let $P_{r, \rho}(\mathcal{F})$ denote the limit when N tends infinity of the percent of forms in $\mathcal{F}_N$ that have at least $r$ normalized zeros on the interval $(-\rho, \rho)$. For $r \geq \mu(\phi, \mathcal{F}) / \phi(\rho)$ and $n = 2m$ is even,
    \begin{equation}
        P_{r, \rho}(\mathcal{F}) \ \leq \ \frac{1_{n \ \rm even}(n-1)!! \sigma_{\phi}^n \pm S(n, a; \phi)}{(r\phi(\rho) - \mu(\phi, \mathcal{F}))^n},
    \end{equation}
    where $\phi$ is an even, non-negative Schwartz test function.
\end{thm}

Previous studies \cite{DM, HM, LM}, which bounded the order of vanishing through $n$-level densities and $n$\textsuperscript{th} centered moments, observed an interesting phenomenon when one increased the level. While it is predicted that higher levels would yield better bounds, one only sees the improvement for higher order vanishing (note this is similar to the effect we observed above in where there must be at least one form with at least one zero sufficiently close). In fact, better bounds for small ranks are produced through lower levels. We see a similar pattern in our bounds resulting from Theorem \ref{evenmainresult}.

\begin{figure}[ht]
    \centering
    \includegraphics[width=10cm]{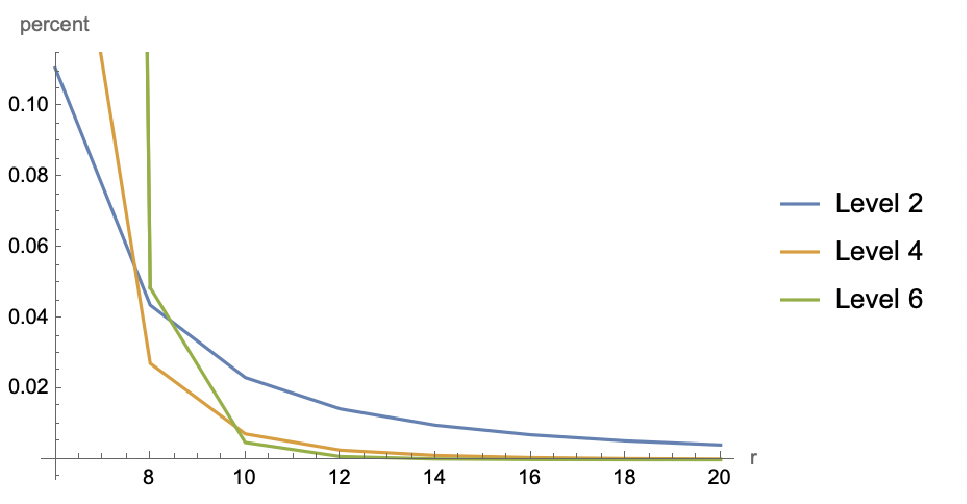}
    \caption{Percentage vs. number of zeros for a fixed $\rho = .4$; for ease of viewing we have connected the discrete data to piece-wise linear curves.}
    \label{percentgraph1}
\end{figure}

If we plot the percentage versus the number of normalized zeros for a fixed interval $(-\rho, \rho)$ (see Figure \ref{percentgraph1}), we see that although the bounds from the higher levels starts off a lot higher than those from lower levels when $r$ is small, they decrease faster and eventually yield better results as $r$ increases (remember the larger $r$ is, the less likely to have at least that many normalized zeros close to the central point). For instance, if we compare the bounds from level 2 to level 6 for a fixed $\rho = .4$, we see that when $r = 6$, the 2-level gives a bound of $0.111085$ (i.e., at most approximately 11\% have a zero within 40\% of the average spacing near the central point) while the 6-level yields $1.58572$ (which is a useless bound as it exceeds a hundred percent!). However, when we look at $r = 20$, we see a significantly better bound from level 6 that gives $0.0000121805$ compared to level 2 that yields $0.0042038$. Thus the higher levels yield greater decay in the probability of the unlikely case of a large number of zeros very close to the central point, far better than the results from the 2-level density.

We also see that for a fixed number of zeros, the higher level densities give better bounds for a larger interval $(-\rho, \rho)$ near the central point (see the tables in \S\ref{tablessection}). For instance, for a fixed $r = 8$, the 4-level density gives a bound of $0.0331395$ at $\rho = .5$, smaller than the 6-level which yields $0.0534808$. However, when we increase $\rho$ to be $.9$, we have the level six giving a better bound of $0.847282$ than $2.14456$ from level four. See Tables \ref{table1} and \ref{table2}.


\section{Preliminaries}

First, we establish the necessary notation and present the required preliminary results, starting with discussing certain properties of the Fourier Transform. As the Fourier transform of a convolution is the product of the Fourier transforms, the complex process of integration is converted to the simpler task of multiplication. However, this comes with the trade-off of needing to compute the inverse Fourier transform later on. The properties below are standard, and are collected for convenience and to identity normalizations; see \cite{SS} for details.

\begin{defi}[Fourier Transform]\label{fouriertransform} The Fourier transform of a function $\phi(x)$ is defined by
    \begin{equation}
        \mathfrak{F}(\phi(x))(y) \ := \ \widehat{\phi}(y) \ = \ \int_{-\infty}^{\infty} \phi(x)e^{-2\pi ixy}dx.
    \end{equation}
\end{defi}

A simple calculation yields the following.

\begin{lem}[Scaling, Derivative]\label{scalingderiv}
    Scaling the Fourier transform by a constant $c$ satisfies
    \begin{equation}
        \mathfrak{F}(\phi(cx))(y) \ = \ \frac{1}{|c|} \cdot \widehat\phi\left(\frac{y}{c}\right),
    \end{equation}  while for derivatives we have
    \begin{equation} \mathfrak{F}(\phi'(x))(y) \ = \ -2\pi i y \widehat\phi(y), \ \ \ {\rm and} \ \ \ \mathfrak{F}(\phi''(x))(y) \ = \  -(2\pi y)^2 \widehat\phi(y).
    \end{equation}
\end{lem}

\begin{defi}[Convolution]\label{defi:convolutionthm}
The convolution of $a(x)$ and $b(x)$ is
    \begin{equation}
        (a \ast b)(x)\ :=\ \int_{-\infty}^{\infty}a(t)b(x-t)dt.
    \end{equation} Note if $a$ and $b$ are functions supported in $(-s, s)$ then $a \ast b$ is supported in $(-2s, 2s)$.
\end{defi}

The following theorem converts convolutions to easier quantities to study.

\begin{thm}[Convolution Theorem]\label{thm:convolutionthm}
For functions $a(x)$ and $b(x)$,
\begin{equation}
    \mathfrak{F}(a \ast b) \ = \ \mathfrak{F}(a) \mathfrak{F}(b).
\end{equation}
\end{thm}

The next lemma shows a nice interplay between convolution and differentiation.

\begin{lem}\label{fourierderivative} The double derivative of the convolution between two functions $a(x)$ and $b(x)$ is
\begin{equation}
    (a \ast b)''(x) \ = \ (a \ast b'')(x).
\end{equation}
\end{lem}

We introduce the Fourier transform pair of the sinc function and the rectangular function.

\begin{defi}[Sinc function]\label{sinc}
    The normalized sinc function is
    \begin{equation}
        \sinc x \ := \
        \begin{cases}
            \frac{\sin \pi x}{\pi x} & \emph{if } x \neq 0 \\
            1 & \emph{if } x = 0.
        \end{cases}
    \end{equation}
\end{defi}

\begin{lem}[Rectangular function]\label{rectangular}
    The rectangular function, also known as the normalized boxcar function, is
    \begin{equation}
        \rect\left(\frac{x}{a}\right) \ := \
        \begin{cases}
            0 & \emph{if } |x| > \frac{a}{2} \\
            \frac{1}{2} & \emph{if } |x| = \frac{a}{2}\\
            1 & \emph{if } |x| < \frac{a}{2}.
        \end{cases}
    \end{equation}
\end{lem}

\begin{lem}\label{sincfourier}
    The sinc function and the rectangular function are the Fourier transform of each other:
    \begin{equation}
        \mathfrak{F}(\sinc(x)) \ = \ \int_{-\infty}^{\infty} \sinc(x)e^{-2\pi ixy}dx \ = \ \rect(x)
    \end{equation}
    and
    \begin{equation}
        \mathfrak{F}\left(\rect\left(\frac{x}{a}\right)\right) \ = \ \int_{-\infty}^{\infty} \rect\left(\frac{x}{a}\right)e^{-2\pi ixy}dx \ = \ a\sinc(a x).
    \end{equation}
\end{lem}

Plancherel's theorem is sometimes useful when integrating a product of two functions, converting the integral of a product to the integral of the product of their Fourier transforms. We use this on $L^2(\mathbb{R})$, which are functions whose integral of the square of their absolute value is finite; as our test functions are bounded and decay rapidly, they are all in this space.

\begin{thm}[Plancherel theorem]\label{Plancherel}
    For two $L^2(\mathbb{R})$ functions $a(x)$ and $b(x)$,
    \begin{equation}
        \int_{-\infty}^{\infty} a(x)\overline{b(x)}dx \ = \ \int_{-\infty}^{\infty} \mathfrak{F}((a(x)) \overline{\mathfrak{F}(b(x))} dx.
    \end{equation}
\end{thm}

In our analysis we have equalities where the left hand side involves sums over zeros, while the right hand side are sums of integrals whose integrands involve products of quantities related to our test functions and the $n$-level density kernels. As these integrals are impossible to evaluate in closed form for our test functions, we use Riemann Sums to approximate the integrals; by choosing the step size sufficiently small we can ensure that the required number of digits of our answer are correct.

\begin{thm}[Riemann Sum] Let $f : [a,b] \to \mathbb{R}$ be a function defined on a closed interval $[a,b]$ and $a, b \in \mathbb{R}$. Let $P = (x_0, x_1, \ldots, x_n)$ be the partition of $[a, b]$ such that
\begin{equation}
    a \ := \ x_0\ < \ x_1\  < \ \cdots\ <\ x_n\ =:\ b.
\end{equation}
The Riemann sum $S$ of $f$ over $[a, b]$ with partition $P$ is
\begin{equation}\label{riemannsum}
    S(f;P) \ := \ \sum_{k = 1}^n f(x_k^\ast)\Delta x_k,
\end{equation}
where $\Delta x_k := x_k - x_{k-1}$ and  $x_k^\ast \in [x_{k-1}, x_k]$. The sum converges to the integral of $f$ over $[a, b]$ as $\max_k \Delta x_k$ approaches zero.
\end{thm}

Many of our terms involve combinatorial factors, including the double factorial.

\begin{defi}[Double Factorial] The double factorial of $n$ is the product of all integers up to $n$ that shares the same parity (odd or even):
\begin{equation}
    n!! \ := \ \prod_{k=0}^{\lceil \frac{n}{2} \rceil - 1} (n-2k) \ = \ n(n-2)(n-4)(n-6)\cdots.
\end{equation}
\end{defi}

\begin{defi} Let $\mathcal{F}_{N, r}^{(\rho)}$ denote the family of forms in $\mathcal{F}_N$ that have at least $r$ zeros within a distance $\rho$. For a fixed $r$, as $N$ tends to infinity through the primes, define $P_{r,\rho}(\mathcal{F})$ to be the limit of the percentage of forms in $\mathcal{F_N}$ that have at least $r$ zeros within a distance $\rho$:
\begin{equation}\label{percent}
    P_{r,\rho}(\mathcal{F}) \ := \ \lim_{N \to \infty} \frac{|\mathcal{F}_{N, r}^{(\rho)}|}{|\mathcal{F}_N|}.
\end{equation}
\end{defi}


\section{Proof of Theorem \ref{oddmainresult}}

\subsection{Construction of test function}\label{construction}
We first outline the construction of the test function $\phi_{\omega}$ that meets the required conditions (we adopt a similar process from \cite{GM} and \cite{HR}).

We start on the Fourier transform side and construct a $\widehat{\phi}(y)$ that has a finite support $(-\frac{\sigma}{n}, \frac{\sigma}{n})$, since it is not easy to find a function that satisfies all the conditions for $\phi$ directly. Let $h$ be an even function that is at least twice continuously differentiable and supported in $[-1, 1]$, and it monotonically decreasing from $0$ to $1$. For a fixed $\omega$ and $\sigma$, let $f(y) := h(\frac{2yn}{\sigma})$ and $g(y) := (f \ast f)(y)$ (the convolution of $f$ with itself \eqref{defi:convolutionthm}); note $\widehat{g}(x) = \widehat{f}(x) \cdot \widehat{f}(x) \ge 0$, and thus $\widehat{g}(x) \ge 0$. Let $\widehat{\phi}_{\omega}(y)$ equal $g(y) + (2 \pi \omega)^{-2} g''(y)$ (where $\widehat{\phi}_{\omega}(y)$ is the Fourier transform of $\phi_{\omega}(x)$).

We show that $\supp\widehat{\phi}_{\omega}(y) \subset (-\sigma/n, \sigma/n)$. Notice, since $\supp(f) \subset (-\sigma/2n, \sigma/2n)$ and $g = f \ast f$, $\supp(g) \subset (-\sigma/n, \sigma/n)$. Since $\widehat{\phi}_{\omega}(y) = g(y) + (2 \pi \omega)^{-2} g''(y)$ and the support of $g''$ is contained in that of $g$, the support of $\widehat{\phi}_{\omega}(y)$ is contained in $(-\frac{\sigma}{n}, \frac{\sigma}{n})$, as desired.

We find $\phi_{\omega}(x)$ by taking the Fourier transform\footnote{A useful property of the Fourier transform is $\mathfrak{F}(\mathfrak{F}(\phi(x))) = \phi(-x)$. Since $\phi$ is even, we can find the original function using this property directly.} of $\widehat{\phi}_{\omega}(y)$. Notice the Fourier transform of $g''(y)$ is $-(2\pi y)^2 \widehat{g}(y)$ and $g'' = f \ast f''$. Combining the above, we have the Fourier transform of $\widehat{\phi}_{\omega}(y) = g(y) + (2 \pi \omega)^{-2} g''(y)$ is $\phi_{\omega}(x) = \widehat{g}(x)\cdot(1 - (x / \omega)^2)$, shown in Figure \ref{phigraph}.

\begin{figure}[ht]
    \centering
    \includegraphics[width=10cm]{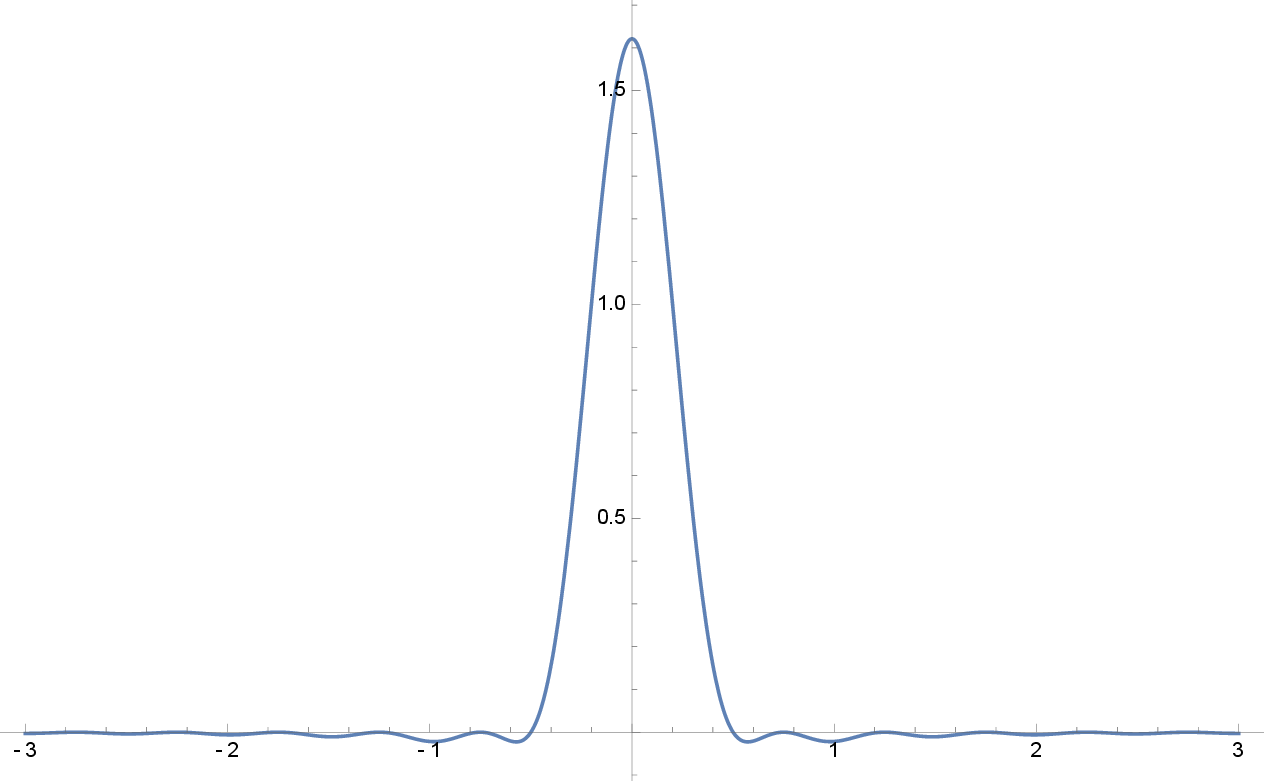}
    \caption{Plot of $\phi_{\omega}(x) = \widehat{g}(x)\cdot(1 - (x / \omega)^2)$, for $h = \cos{\left(\frac{\pi y}{2}\right)}$, $\sigma = 2$, $\omega = .5$, and $n = 1$.}
    \label{phigraph}
\end{figure}

\begin{rem}
    Following the framework of constructing the test function above, one can adopt more elaborate constructions, such as adding a quartic term in $\phi$, to obtain potentially better test functions that yield better results, but at the cost of more involved computations.
\end{rem}

\subsection{Proof of Theorem \ref{oddmainresult}} We use odd centered moments to obtain a lower bound on the number of normalized zeros within a small distance $\omega$ from the central point. Then we prove the theorem by showing a contradiction if there are no zeros in the interval $(-\omega, \omega)$.

As remarked in the introduction, there is a long history of results on the $n$-level densities and the combinatorially easier centered moments. The first results are due to Iwaniec-Luo-Sarnak \cite{ILS}, who computed the $1$-level density with support up to $2$. Their result  was later extended by Hughes-Miller \cite{HM} to the $n$-level densities and centered moments with support up to $1/(n-1)$. Recently there was a breakthrough by Cohen et. al. \cite{C--}, who were finally able to handle the combinatorics and prove the $n$-level densities and centered moments agree with random matrix theory for support up to $2/n$; see Theorem \ref{bound}. We extract the following useful consequence of their work.

\begin{cor} We can write \eqref{RHSlimit} as
   \begin{equation}\label{threelevel}
       \lim_{\substack{N\to\infty \\ N \text{\rm prime}}} \frac{1}{|\mathcal{F}_{N}|}\sum_{f \in \mathcal{F}_{N}} \left( \sum_{j}\phi(\widetilde{\gamma}_{f, j}c_{n}) -  \mu(\phi, \mathcal{F}) \right)^{n} \ = \ 1_{n \ \rm even}(n-1)!! \sigma_{\phi}^n \pm S(n, a; \phi),
   \end{equation}
    where $\mu(\phi, \mathcal{F})$ is the mean of the family of cuspidal newforms.
   \begin{proof}
   	    The difference between the mean of the family in the limit and the mean when the level is $N$ is negligible; this follows from an application of the Binomial Theorem and the Cauchy-Schwartz inequality. We can therefore simplify the calculations and replace $\mu(\phi, \mathcal{F}_{N})$ with $\mu(\phi, \mathcal{F})$. For a complete proof see \cite{DM}.
   \end{proof}
\end{cor}

\begin{thm}[See \cite{ILS, HM}]\label{mean}
     The mean over SO(even) is
    \begin{equation}
        \mu(\phi, \mathcal{F}) \ := \ \widehat{\phi}(0) + \frac{1}{2}\int_{-\infty}^{\infty}\widehat{\phi}(y)dy.
    \end{equation}
\end{thm}

We restate our first main result, Theorem \ref{oddmainresult}, for convenience, and then give the proof. \\ \

\noindent \textbf{Theorem \ref{oddmainresult}: Assume GRH. For $n$ odd, whenever $\omega$ satisfies the inequality
        \begin{equation}
        -\left( \widehat\phi_{\omega}(0) + \frac{1}{2}\int_{-\sigma/n}^{\sigma/n}\widehat{\phi}_{\omega}(y)dy\right)^{n} < \ S(n, a; \phi_{\omega}),
        \end{equation}
    there exists at least one form with one normalized zero on the interval $(-\omega, \omega)$.}

\begin{proof}
    As remarked earlier we focus only the family of even forms; there are trivial modifications if one studies the odd forms (since all forms with an odd functional equation are guaranteed a zero at the central point, we would need to remove that contribution).

    We start with equation \eqref{threelevel}. As $\phi_{\omega}$ is non-negative when $|x| \leq \omega$ and non-positive when $|x| > \omega$, the contribution of the scaled zeroes to the centered moment when $|x| \geq \omega$ is non-positive. Since the power is odd, if we drop those non-positive contributions, it will not increase the left-hand side; \emph{this is why we restrict to odd levels}. Thus the centered moment gives a lower bound
    \begin{equation}
        \lim_{\substack{N\to\infty \\ N \text{\rm prime}}} \frac{1}{|\mathcal{F}_{N}|}\sum_{f \in \mathcal{F}_{N}} \left( \sum_{|\gamma_{f, j}| \leq \omega}\phi_{\omega}(\widetilde{\gamma}_{f, j}c_{n}) -  \mu(\phi_{\omega}, \mathcal{F}) \right)^{n} \ \geq \  S(n, a; \phi_{\omega}).
    \end{equation}
    Assume no forms have a normalized zero in the interval $(-\omega, \omega)$. Then the scaled zeroes with imaginary part at most $\omega$ in absolute value contribute 0 to the centered moment, and thus
    \begin{equation}
        \lim_{\substack{N\to\infty \\ N \text{\rm prime}}} \frac{1}{|\mathcal{F}_{N}|}\sum_{f \in \mathcal{F}_{N}} \left( -  \mu(\phi_{\omega}, \mathcal{F}) \right)^{n} \ \geq \  S(n, a; \phi_{\omega}).
    \end{equation}
    Since the mean is independent of $f$, we can extract it from the summation:
    \begin{equation}
        \left( -\mu(\phi_{\omega}, \mathcal{F}) \right)^{n} \lim_{\substack{N\to\infty \\ N \text{\rm prime}}} \frac{1}{|\mathcal{F}_{N}|}\sum_{f \in \mathcal{F}_{N}} 1 \ \geq \  S(n, a; \phi_{\omega}).
    \end{equation}
    As $\lim_{\substack{N\to\infty \\ N \text{\rm prime}}} \frac{1}{|\mathcal{F}_{N}|}\sum_{f \in \mathcal{F}_{N}} 1 = \lim_{\substack{N\to\infty \\ N \text{\rm prime}}} \frac{|\mathcal{F}_{N}|}{|\mathcal{F}_{N}|} = 1$, we find
    \begin{equation}
        \left( - \mu(\phi_{\omega}, \mathcal{F}) \right)^{n} \ \geq \  S(n, a; \phi_{\omega}).
    \end{equation}
    Because of the compact support of $\widehat\phi_{\omega}$, by Theorem \ref{mean} this is
    \begin{equation}
        -\left( \widehat\phi_{\omega}(0) + \frac{1}{2}\int_{-\sigma/n}^{\sigma/n}\widehat\phi_{\omega}(y)dy\right)^{n} \geq \  S(n, a; \phi_{\omega}).
    \end{equation}
    Thus, if we choose $\omega$ that satisfies the following inequality
    \begin{equation}\label{inequality}
        -\left( \widehat\phi_{\omega}(0) + \frac{1}{2}\int_{-\sigma/n}^{\sigma/n}\widehat\phi_{\omega}(y)dy\right)^{n} < \  S(n, a; \phi_{\omega}),
    \end{equation}
     we get a contradiction, indicating there has to be at least one form with one normalized zero on the interval $(-\omega, \omega)$.
\end{proof}

\begin{rem}
The theorem above improves itself as the zeros are distributed symmetrically along the critical line, and thus if $1/2 + i \gamma$ is a zero so too is $1/2 - i \gamma$. In fact, there are at least two zeros in such a window.
\end{rem}

\subsection{Explicit bounds on first zero}
We examine the inequality in \eqref{inequality} and give an explicit formula for the one-level calculation of the bound $\omega$.

\begin{thm} Assume GRH.
    The first normalized zero of the family of cuspidal newforms is in the interval $(\omega_{\min}, \omega_{\min})$, where
    \begin{equation}
        \omega_{\min}(\sigma, h) \ > \ \left( -\frac{\sigma \int_{0}^{1}h(u)^{2}\,du + \frac{\sigma^2}{4}\int_{-1}^{1}\int_{0}^{2/\sigma} h(u) h(v - u) \,dv \,du}{\frac{1}{\sigma}\int_{0}^{1} h(u)h''(u) \,du + \frac{1}{4}\int_{-1}^{1}\int_{0}^{2/\sigma} h(u) h''(v - u) \,dv \,du}\right)^{-\frac{1}{2}} \pi^{-1}.
    \end{equation}
\end{thm}

\begin{proof}
Note that $\int_{-\sigma/n}^{\sigma/n}\widehat\phi_{\omega}(y)dy = \phi_{\omega}(0)$. Thus for $n = 1$, we can rewrite the inequality in \eqref{inequality} as
    \begin{equation}\label{LHSsimplified}
        - \widehat\phi_{\omega}(0) - \frac{1}{2}\phi_{\omega}(0) \ < \ S(1, 1; \phi_{\omega}).
    \end{equation}
    By Theorem \ref{bound}, the right hand side of \eqref{LHSsimplified} is
    \begin{equation}
        S(1, 1; \phi_{\omega}) \ = \ R(1, 1; \phi_{\omega}) \ = \ -\frac{1}{2}\phi_\omega(0) + \int_{-\infty}^{\infty}\phi_\omega(x)\frac{\sin(2\pi x)}{2\pi x}\,dx.
    \end{equation}
    The $-\frac12\phi_\omega(0)$ terms cancel, so we focus simplifying the second term. Notice that directly integrating against the sine function is hard, so instead we use the Plancherel Theorem (Theorem \ref{Plancherel}) and rewrite the integral as an integral of the product of their Fourier transforms. The Fourier transform of $\phi_\omega(x)$ is $\widehat\phi_\omega(y)$. To perform the integration of $\phi_\omega(x)$ against $\sin(2\pi x)/2\pi x$, note the latter is the $\sinc$ function from Definition \ref{sinc}. Since the Fourier transform of the sinc function is the rectangular function, with some re-scaling by Theorem \ref{scalingderiv} we get
    \begin{equation}
        \int_{-\infty}^{\infty}\phi_\omega(x)\frac{\sin(2\pi x)}{2\pi x}\,dx \ = \ 2\int_{0}^{\infty}\widehat\phi_\omega(y)\frac{1}{2}\rect\left(\frac{y}{2}\right)\,dy.
    \end{equation}
    Notice the rectangular function $\rect\left(\frac{y}{2}\right)$ equals $1$ from $-1$ to $1$ and $0$ everywhere else, so we may rewrite the above integral as $\int_0^1\widehat\phi_\omega(y)dy$. Thus the RHS simplifies to $-\frac{1}{2}\phi_\omega(0) + \int_0^1\widehat\phi_\omega(y)dy$. Combining with what we have for the LHS from \eqref{LHSsimplified}, we get that the bound for the smallest first normalized zero is when
    \begin{equation}\label{phiequality}
         - \widehat\phi_{\omega}(0) \ < \ \int_0^1\widehat\phi_\omega(y)\,dy.
    \end{equation}
    We now simplify $\widehat\phi_\omega(0)$ in terms of an arbitrary $h$ that fulfills our conditions. From the definition of $\widehat\phi_{\omega}$ we have
    \begin{equation}
        \widehat\phi_{\omega}(0) \ = \ g(0) + (2 \pi \omega)^{-2}g''(0).
    \end{equation}
    Through the construction of $g$ and by taking advantage of the evenness of $h$ and $f$, we get
    \begin{equation}
        g(0)\ =\ \int_{-\infty}^{\infty}f(t)f(0 - t)dt = 2\int_{0}^{\sigma/2}h\left(\frac{2t}{\sigma}\right)^{2}dt \ =\ \sigma\int_{0}^{1}h(u)^{2}du.
    \end{equation}
    As $g = f \ast f$ implies $g'' = f \ast f''$ we find
    \begin{align}
        g''(0) & \ = \ \int_{-\sigma/2}^{\sigma/2} f(t)f''(0 - t) \,dt\nonumber\\
			     & \ = \ 2\int_{0}^{\sigma/2} f(t)f''(t) \,dt\nonumber\\
			     & \ = \ 2\int_0^{\sigma/2} h\left(\frac{2t}{\sigma}\right) \frac{4}{\sigma^2}h''\left(\frac{2t}{\sigma}\right) \,dt\nonumber\\	
                 & \ = \ \frac{4}{\sigma}\int_{0}^{1} h(u)h''(u) \,du.
    \end{align}
    Combining all of the above, we find
    \begin{equation}
        -\widehat\phi_\omega(0) \ = \ -\sigma\int_{0}^{1}h(u)^{2}\,du - (2 \pi \omega)^{-2}\frac{4}{\sigma}\int_{0}^{1} h(u)h''(u) \,du.
    \end{equation}
    We now simplify the RHS of \eqref{phiequality}. By the expansion of $\widehat\phi_{\omega}$,
    \begin{equation}
        \int_0^1\widehat\phi_\omega(y)\,dy \ = \ \int_0^1 g(y)\,dy + (2\pi\omega)^{-2}\int_0^1 g''(y)\,dy.
    \end{equation}
    To rewrite the integrals in terms of $h$, we detail the procedure to simplify the first term from above; the second term follows analogously\footnote{The Fubini-Tonelli theorem states that we can change the order of integration for a two-dimensional integral if it yields a finite values when the integrand is replaced by its absolute value. Since the function we are using is finitely supported, using this theorem is justified.}.

    \begin{align}
        \int_0^1 g(y)\,dy & \ = \ \int_0^1 \int_{-\infty}^{\infty}f(t)f(y-t) \,dt \,dy \nonumber \\
        & \ = \ \int_0^1 \int_{-\sigma/2}^{\sigma/2} h\left( \frac{2t}{\sigma}\right) h\left( \frac{2(y-t)}{\sigma} \right) \,dt \,dy\nonumber \\
        & \ = \ \frac{\sigma}{2}\int_0^1 \int_{-1}^{1} h(u) h\left( \frac{2y}{\sigma} - u\right) \,du \,dy \nonumber \\
        & \ = \ \frac{\sigma}{2}\int_{-1}^{1}\int_0^1 h(u) h\left( \frac{2y}{\sigma} - u\right) \,dy \,du \nonumber \\
        & \ = \ \frac{\sigma^2}{4}\int_{-1}^{1}\int_{0}^{2/\sigma} h(u) h(v - u) \,dv \,du.
    \end{align}
    We go through the same procedure to simplify the second term. Eventually, the right-hand side becomes
    \begin{equation}
        \int_0^1\widehat\phi_\omega(y) \,dy \ = \ \frac{\sigma^2}{4}\int_{-1}^{1}\int_{0}^{2/\sigma} h(u) h(v - u) \,dv \,du + (2 \pi \omega)^{-2}\int_{-1}^{1}\int_{0}^{2/\sigma} h(u) h''(v - u) \,dv \,du.
    \end{equation}
    Putting the simplified versions of two sides of \eqref{phiequality} together and isolating $\omega$, we get the expression for the minimum $\omega$:
    \begin{equation}
        \omega \ > \ \left( -\frac{\sigma \int_{0}^{1}h(u)^{2}\,du + \frac{\sigma^2}{4}\int_{-1}^{1}\int_{0}^{2/\sigma} h(u) h(v - u) \,dv \,du}{\frac{1}{\sigma}\int_{0}^{1} h(u)h''(u) \,du + \frac{1}{4}\int_{-1}^{1}\int_{0}^{2/\sigma} h(u) h''(v - u) \,dv \,du}\right)^{-\frac{1}{2}} \pi^{-1}.
    \end{equation}
\end{proof}

\begin{rem}
    Since number theory has only been able to compute the needed quantities for $\sigma$ up to $2$ (under the assumption of GRH), we get $\omega_{\min}(2, h) \approx 0.25$ for $h = \cos(\pi y/2)$ is the optimum function that yields the smallest bound. We also looked at results for level 3 and 5 by directly integrating both sides of the inequality, hoping to see better bounds than the one-level. Since the right-hand side involves a multi-dimensional integral of a very complicated function ($\phi_{\omega}$, involving Fourier transforms and convolutions), we estimate it through Riemann Sums, see \ref{riemannsum}. Although it cannot give exact values, the approximation is sufficiently close as we have a small error term. For our evaluation, we chose to use the simple function $h(y) = 1 - y^2$ to construct our test function $\phi_\omega$. This is because the complex nature of the right-hand side makes evaluation extremely slow, and since optimizing the function has little improvement in the final results, we are satisfied with using the chosen $h$ for now (though ongoing studies will focus on finding the optimal function to use). However, although we expect higher levels to yield better bounds for support sufficiently large, that is not the case here. As we increase the level, the bound gets worse for support $\sigma = 2$: for $n = 3$, $\omega \approx .34$, and for $n = 5$, $\omega \approx .85$.
\end{rem}

\section{Proof of Theorem \ref{evenmainresult}}

We obtain an upper bound through even centered moments and using the naive test function $\phi$ defined in \eqref{naive}.

\begin{thm} Assume GRH.
Let $P_{r,\rho}(\mathcal{F})$ denote the percent of forms in the family that has at least $r$ normalized zeros on the interval $(-\rho, \rho)$ near the central point. For an even $n$ and $r \geq \mu(\phi, \mathcal{F}) / \phi(\rho)$,
    \begin{equation}
        P_{r,\rho}(\mathcal{F}) \ \leq \ \frac{1_{n \ \rm even}(n-1)!! \sigma_{\phi}^n + S(n, a; \phi)}{(r\phi(\rho) - \mu(\phi, \mathcal{F}))^n}.
    \end{equation}
\end{thm}
\begin{proof}
    We first rewrite \eqref{threelevel} as
    \begin{equation}
        \lim_{\substack{N\to\infty \\ N \text{\rm prime}}} \frac{1}{|\mathcal{F}_{N}|}\sum_{f \in \mathcal{F}_{N}} \left( \sum_{|\widetilde{\gamma}_{f, j}| \leq \rho}\phi(\widetilde{\gamma}_{f, j}c_{n}) + T_f(\phi) -  \mu(\phi, \mathcal{F}) \right)^{n} \ = \ 1_{n \ \rm even}(n-1)!! \sigma_{\phi}^n + S(n, a; \phi),
    \end{equation}
    where $T_f(\phi)$ is the contribution from the scaled $\phi(\gamma_{f, j}c_{n})$ when $|\gamma_{f, j}| > \rho$.

    We look at forms with at least $r$ normalized zeros within a distance $\rho$ from the central point, and denote the set of these forms by $\mathcal{F}_{N,r}^{(\rho)}$. Since $n$ is even, by dropping all the form with less than $r$ normalized zeros in the interval $(-\rho, \rho)$, we are dropping a non-negative sum; \emph{this is why we work with an even moment}. Since doing so cannot increase the size of the left-hand side, we obtain the following upper bound:
    \begin{equation}
        \lim_{\substack{N\to\infty \\ N \text{\rm prime}}} \frac{1}{|\mathcal{F}_{N}|}\sum_{f \in \mathcal{F}_{N,r}^{(\rho)}} \left( \sum_{|\widetilde{\gamma}_{f, j}| \leq \rho}\phi(\widetilde{\gamma}_{f, j}c_{n}) + T_f(\phi) -  \mu(\phi, \mathcal{F}) \right)^{n} \ \leq \ 1_{n \ \rm even}(n-1)!! \sigma_{\phi}^n + S(n, a; \phi).
    \end{equation}
    We then replace the summation of $\phi(\widetilde{\gamma}_{f, j}c_{n})$ with $r\phi(\rho)$, only counting the contribution of $r$ zeros. This involves several steps, and importantly none of these lead to an increase in the LHS. This is because $\phi$ is non-negative, $n$ is even, and since $\phi$ is decreasing from $(0, \rho)$, if we replace the contribution of the test function at these zeros by $\phi(\rho)$ we do not increase the sum. It is important to note that $r$ is an even number, since by symmetry, if $1/2 + i\gamma$ is a zero, so too is $1/2 - i\gamma$.

    We also want to be able to drop the term $T_f(\phi)$ and not increase left-hand side. However, although $\phi$ is non-negative, this is not the case if the sum of the first two terms in the parentheses are smaller than the absolute value of the mean. Thus we restrict $r$ to be greater than or equal to $\mu(\phi, \mathcal{F}) / \phi(\rho)$, and further decrease the LHS by dropping $T_f(\phi)$:
    \begin{equation}
        \lim_{\substack{N\to\infty \\ N \text{\rm prime}}} \frac{1}{|\mathcal{F}_{N}|} \sum_{f \in \mathcal{F}_{N,r}^{(\rho)}} \left(r\phi(\rho) -  \mu(\phi, \mathcal{F}) \right)^{n} \ \leq \ 1_{n \ \rm even}(n-1)!! \sigma_{\phi}^n + S(n, a; \phi).
    \end{equation}
    As the summand is independent of $f$, we move it outside the sum and find
    \begin{equation}
         \left(r\phi(\rho) -  \mu(\phi, \mathcal{F}) \right)^{n} \lim_{\substack{N\to\infty \\ N \text{\rm prime}}} \frac{1}{|\mathcal{F}_{N}|}\sum_{f \in \mathcal{F}_{N,r}^{(\rho)}} 1 \ \leq \ 1_{n \ \rm even}(n-1)!! \sigma_{\phi}^n + S(n, a; \phi).
    \end{equation}
    Notice, by \eqref{percent}, $\lim_{\substack{N\to\infty \\ N \text{\rm prime}}} \frac{1}{|\mathcal{F}_{N}|}\sum_{f \in \mathcal{F}_{N,r}^{(\rho)}} 1 = \lim_{\substack{N\to\infty \\ N \text{\rm prime}}}\frac{|\mathcal{F}_{N,r}^{(\rho)}|}{|\mathcal{F}_{N}|} = P_{r,\rho}(\mathcal{F})$. Isolating the percentage yields
    \begin{equation}
        P_{r,\rho}(\mathcal{F}) \ \leq \ \frac{1_{n \ \rm even}(n-1)!! \sigma_{\phi}^n + S(n, a; \phi)}{\left(r\phi(\rho) -  \mu(\phi, \mathcal{F}) \right)^{n} }.
    \end{equation}
\end{proof}

\subsection{Explicit bounds for percentages}\label{tablessection} We provide explicit bounds calculated using Theorem \ref{evenmainresult}. Notice that all number of zeros are even by symmetry.

\begin{center}
\begin{table}[ht]
\caption{\label{table1} Probability of forms having a given number of normalized zeros $0.2$ away from the central point, using the naive test function with a support of $2/n$. "N/A" means that the level cannot give the bound because of the restriction in Theorem \ref{evenmainresult}}.
\begin{tabular}{ |p{3.2cm}||p{3cm}|p{3cm}|p{3cm}|  }
 \hline
 Number of zeros & 2-level & 4-level & 6-level \\
 \hline
 2 & 6.651738 & N/A & N/A \\
 \hline
 4 & 0.104108 & 2.370419 & 2.147231$\cdot 10^4$ \\
 \hline
 6 & 0.029617 & 0.069998 & 0.809927 \\
 \hline
 8 & 0.013769 & 0.011079 & 0.022517 \\
 \hline
 10 & 0.007924 & 0.003152& 0.002427 \\
 \hline
 12 & 0.005142 & 0.001213 & 4.79813$\cdot 10^{-4}$  \\
 \hline
 14 & 0.003605 & 5.612212$\cdot 10^{-4}$ & 1.340970$\cdot 10^{-4}$\\
 \hline
 16 & 0.002666 & 2.942389$\cdot 10^{-4}$ & 4.688515$\cdot 10^{-5}$ \\
 \hline
 18 & 0.002052 & 1.687747$\cdot 10^{-4}$ & 1.917773$\cdot 10^{-5}$ \\
 \hline
 20 & 0.001627 & 1.036100$\cdot 10^{-4}$ & 8.809943$\cdot 10^{-6}$ \\
 \hline
\end{tabular}
\end{table}
\end{center}

\begin{center}
\begin{table}[ht]
\caption{\label{table2}Probability of forms having a given number of normalized zeros $0.4$ away from the central point, using the naive test function with a support of $2/n$.}.
\begin{tabular}{ |p{3.2cm}||p{3cm}|p{3cm}|p{3cm}|  }
 \hline
 Number of zeros & 2-level & 4-level & 6-level \\
 \hline
 4 & 0.665694 & 8.334733 & 1.744392$\cdot 10^3$ \\
 \hline
 6 & 0.111085 & 0.145883 & 1.585718 \\
 \hline
 8 & 0.043857 & 0.020351 & 0.036592 \\
 \hline
 10 & 0.023310 & 0.005469 & 0.003680 \\
 \hline
 12 & 0.014430 & 0.002038 & 0.000702 \\
 \hline
 14 & 0.009804 & 0.000924 & 0.000192 \\
 \hline
 16 & 0.007093 & 0.000475 & 6.606784$\cdot 10^{-5}$ \\
 \hline
 18 & 0.005369 & 0.000271 & 2.673289$\cdot 10^{-5}$ \\
 \hline
 20 & 0.004204 & 0.000165 & 1.218053$\cdot 10^{-5}$ \\
 \hline
\end{tabular}
\end{table}
\end{center}

As shown in both Table \ref{table1} and \ref{table2}, we see the pattern that for bounding a small number of zeros, the 2-level does a significantly better job than the 4 or 6 level. In fact, both bounds for the 4 and 6 level when $r = 4$ are useless as they already exceed a hundred percent! Moreover, levels 4 and 6 cannot even produce bounds for $r = 1$ because of the restriction of $r \geq \mu(\phi, \mathcal{F}) / \phi(\rho)$ in our theorem. However, as we look at a larger number of zeros in the interval, the 4-level does eventually yield a better bound than the 2-level, and similarly eventually the 6-level beats both the 2-level and the 4-level in giving the smallest bound.

\begin{center}
\begin{table}[ht]
\caption{\label{table3} Probability of forms having a given number of normalized zeros $0.8$ away from the central point, using the naive test function with a support of $2/n$. "N/A" means that the level cannot give the bound because of the restriction in Theorem \ref{evenmainresult}.}
\begin{tabular}{ |p{3.2cm}||p{3cm}|p{3cm}|p{3cm}|  }
 \hline
 Number of zeros & 2-level & 4-level & 6-level \\
 \hline
 6 & N/A & 10.849910 & 48.154279 \\
 \hline
  & &  &  \\
 \hline
 16 & N/A & 0.004235 & 2.83230$\cdot 10^{-4}$\\
 \hline
  & &  &  \\
 \hline
 26 & N/A & 3.541901$\cdot 10^{-4}$  & 6.716802$\cdot 10^{-6}$ \\
 \hline
 28 & 420.045063 & 2.486819$\cdot 10^{-4}$ & 3.943864$\cdot 10^{-6}$  \\
 \hline
 30 & 20.991406 & 1.796948$\cdot 10^{-4}$  & 2.418466$\cdot 10^{-6}$ \\
 \hline
 32 & 6.651738 & 1.330555$\cdot 10^{-4}$ &  1.538761$\cdot 10^{-6}$ \\
 \hline
 34 & 3.220871 & 1.006126$\cdot 10^{-4}$ &  1.010576$\cdot 10^{-6}$ \\
 \hline
\end{tabular}
\end{table}
\end{center}

We also see the pattern that for the same number of zeros, the chance of them existing in a smaller range $(-\rho, \rho)$ is much lower than in a larger interval. For instance, using $n=2$ the probability of having $8$ zeros on the interval $(-.2, .2)$ is approximately $0.013769$, while for the interval $(-.4, .3)$ it is much larger, about $0.043857$.

Higher levels or moments are also better at bounding the number of normalized zeros within a larger distance from the central point. Because of the restriction of $r \geq \mu(\phi, \mathcal{F}) / \phi(\rho)$, lower levels are simply incapable of bounding small number of zeros! For example, in Table \ref{table3} when we fixed a bigger $\rho = .8$, the 2-level can only non-trivially bound the number of zeros when $r > 27$, while the 4-level and 6-level are perfectly fine with $r$ as small as $6$. We also see the 6-level producing much better bounds than smaller levels in Table \ref{table3}. Thus, even though the higher levels are much more difficult to calculate, the tradeoff is worthwhile as we see its significance when we increase the interval $(-\rho, \rho)$ as well as the number of zeros $r$.


\section{Future work}
In Theorem \ref{oddmainresult} we focused only on the even forms of the family, as the odd forms are guaranteed to have one zero on the central point. Our result can be extended to odd forms by simply adding that contribution.

There are several directions one can take to improve these results. One is to optimize the test function. Though the optimum test function for the one-level density was determined in previous studies, optimizing test function for higher level densities is open. Building off of the framework of the construction of the test function in Theorem \ref{oddmainresult}, it is possible to alter it by adding a quartic term, for instance, to see improvement on bounds for higher levels. However, previous work highlights that the greatest gains come from increasing the support and using higher $n$-level, which is what we have focused on in this paper. At the cost of some additional standard computations, one can easily add additional terms and obtain similar integrals to approximate; initial explorations indicated that the gain was minimal, and thus we focused on the significant improvements that arise with higher level densities and moments, justifying the additional work required.

For higher levels, one can also try to simplify the limit calculation through complex analysis, making evaluation easier for explicit bounds.



\ \\
\end{document}